\documentclass[pdflatex,sn-mathphys-num]{sn-jnl}% Math and Physical Sciences Numbered Reference Style
\usepackage{graphicx}%
\usepackage{multirow}%
\usepackage{amsmath,amssymb,amsfonts}%
\usepackage{amsthm}%
\usepackage{mathrsfs}%
\usepackage[title]{appendix}%
\usepackage{xcolor}%
\usepackage{textcomp}%
\usepackage{manyfoot}%
\usepackage{booktabs}%
\usepackage{algorithm}%
\usepackage{algorithmicx}%
\usepackage{algpseudocode}%
\usepackage{listings}%
\usepackage{relsize}
\usepackage{colonequals}

\theoremstyle{thmstyleone}%
\newtheorem{theorem}{Theorem}%  meant for continuous numbers
\newtheorem{lemma}[theorem]{Lemma}

\newtheorem{proposition}[theorem]{Proposition}% 

\theoremstyle{thmstyletwo}%
\newtheorem{remark}{Remark}%

\theoremstyle{thmstylethree}%

\newcommand{\N}{\mathbb{N}}
\newcommand{\Z}{\mathbb{Z}}

\newcommand{\R}{\mathbb{R}}
\newcommand{\C}{\mathbb{C}}

\newcommand{\ve}{\varepsilon}
\newcommand{\on}{\operatorname}

\newcommand{\of}[1]{\left(#1\right)}
\newcommand{\set}[1]{\left\{#1\right\}}
\newcommand{\BEu}[1]{\underset{#1}{\mathlarger{\mathlarger{\mathbb{E}}}^{~}}\,}
\newcommand{\BEul}[1]{\underset{#1}{\mathlarger{\mathlarger{\mathbb{E}}}^\log}\,}
\newcommand{\BEulk}[1]{\underset{#1}{\mathlarger{\mathlarger{\mathbb{E}}}^{\log^{k-1}}}\,}

\begin{document}

\title[A logarithmic analogue of Alladi's formula]{A logarithmic analogue of Alladi's formula}


\author{\fnm{Biao} \sur{Wang}}\email{bwang@ynu.edu.cn}

\affil{\orgdiv{School of Mathematics and Statistics}, \orgname{Yunnan University}, \orgaddress{\city{Kunming}, \postcode{650500}, \state{Yunnan}, \country{China}}}


\abstract{Let $\mu(n)$ be the M\"{o}bius function. Let $P^-(n)$ denote the smallest prime factor of an integer $n$. In 1977, Alladi established the following formula related to the prime number theorem for arithmetic progressions
\[
	-\sum_{\substack{n\geq 2\\ P^-(n)\equiv \ell ({\rm mod}k)}}\frac{\mu(n)}{n}=\frac1{\varphi(k)}
\]
for positive integers $\ell, k\ge1$ with $(\ell,k)=1$, where $\varphi$ is Euler's totient function. In this note, we will show a logarithmic analogue of Alladi's formula in an elementary proof.}




\maketitle

\section{Introduction}

Let $\mu(n)$ be the M\"{o}bius function defined by 
\[
\mu(n)=\begin{cases}
	1 & \text{if } n=1,\\
	0 & \text{if } p^2\mid n \text{ for some prime } p,\\
	(-1)^r & \text{if } n=p_1\cdots p_r \text{ where the } p_i \text{ are distinct primes}.
\end{cases}
\]
Then it is well-known (e.g., \cite{Mangoldt1897}) that the prime number theorem (PNT) is equivalent to the assertion that 
\begin{equation}\label{PNT}
	\sum_{n=1}^\infty\frac{\mu(n)}{n}=0.
\end{equation}
By reordering \eqref{PNT}, it is equivalent to asserting that $-\sum_{n=2}^\infty\frac{\mu(n)}{n}=1$, which can be interpreted as the statement that the probability of integers $n\ge2$ being divisible by a prime is $1$.  Let $P^-(n)$ and $P^+(n)$ be the smallest prime factor of $n$ and the largest prime factor of $n$ for $n\ge2$, respectively. Let $P^-(1)=P^+(1)=1$. In 1977, Alladi  \cite{Alladi1977}  introduced a duality between $P^-(n)$ and  $P^+(n)$, showing that for any function $f$ defined on integers  with $f(1)=0$, we have
\begin{align}
	\sum_{d|n}\mu(d)f(P^-(d))&=-f(P^+(n)),\label{duality1}\\
	\sum_{d|n}\mu(d)f(P^+(d))&=-f(P^-(n)).\label{duality2}
\end{align}
Let $k,\ell \geq 1$ be integers with $(\ell,k)=1$. Then applying the PNT for arithmetic progressions with \eqref{duality1}, Alladi \cite{Alladi1977} established that
\begin{equation}\label{alladi}
	-\sum_{\substack{n\geq 2\\ P^-(n)\equiv \ell (\on{mod}k)}}\frac{\mu(n)}{n}=\frac1{\varphi(k)},
\end{equation}
where $\varphi$ is Euler's totient function. 

In the last few years, Alladi's formula \eqref{alladi} has been greatly generalized in various settings in number theory, see \cite{Dawsey2017, OnoSchneiderWagner2017, SweetingWoo2019, KuralMcDonaldSah2020, Wang2021jnt, OnoSchneiderWagner2021, DuanWangYi2021, DuanMaYi2022}. Recently, Alladi and Johnson \cite{AlladiJohnson2024} proved that
\begin{equation}\label{AlladiJohnson2024}
	\sum_{\substack{n\geq 2\\ P^-(n)\equiv \ell (\on{mod}k)}}{\mu(n)\omega(n)\over n}=0,
\end{equation}
where $\omega(n)$ is the number of distinct prime factors of $n$. Later, \eqref{AlladiJohnson2024} is generalized by Tenenbaum \cite{Tenenbaum2024} and Sengupta \cite{Sengupta2025}. In 1899, Landau \cite{Landau1899}\footnote{See \cite{Coons2010} for an English translation by Michael J. Coons.} proved in his doctoral dissertation that  the PNT is also equivalent to the following logarithmic analogue of \eqref{PNT}:
\begin{equation}\label{log_PNT}
	\lim_{x\to\infty} \frac{1}{\log x} \sum_{n\leq x}\frac{\mu(n)\log n}{n}=0.
\end{equation}
Inspired by \eqref{AlladiJohnson2024} and \eqref{log_PNT}, we will show a logarithmic analogue of  Alladi's formula in this note.

Let $S$ be a set of primes. We say $S$ has a \textit{natural density} if the following limit  $$\lim_{x\to\infty} \frac{\#\set{p\in S: p\leq x}}{\#\set{p \text{ prime}: p\leq x}}$$ exists. The following theorem is our main result.

\begin{theorem}\label{mainthm_logAlladi}
		If $S$ is a set of primes of natural density, then we have
	\begin{equation}\label{logAlladi_min}
			\lim_{x\to \infty}\frac{1}{\log x}\sum_{\substack{n\leq x\\P^-(n)\in S}} \frac{\mu(n)\log n }{n}=0.
	\end{equation}	
Moreover,
	\begin{equation}\label{logAlladi_max}
			\lim_{x\to \infty}\frac{1}{\log x}\sum_{\substack{n\leq x\\P^+(n)\in S}} \frac{\mu(n)\log n }{n}=0.
	\end{equation}

\end{theorem}

Taking $S$ to be the set of all primes, then we recover \eqref{log_PNT} from \eqref{logAlladi_min} or \eqref{logAlladi_max}. Hence Theorem~\ref{mainthm_logAlladi} is a refinement of the equivalent form \eqref{log_PNT} of the PNT. If we take $S=\set{p \text{ prime}: p\equiv \ell (\on{mod}k)}$ for $(\ell,k)=1$, then we get the following analogue of \eqref{AlladiJohnson2024}:
\begin{equation}
	\lim_{x\to \infty}\frac{1}{\log x}\sum_{\substack{n\leq x\\P^-(n)\equiv \ell (\on{mod}k)}} \frac{\mu(n)\log n }{n}=0.
\end{equation}

To prove Theorem~\ref{mainthm_logAlladi}, we will mainly show a logarithmic analogue of a duality theorem obtained by Alladi \cite{Alladi1977} in the next section.

\section{Duality between $P^+(n)$ and $P^-(n)$}

Based on the duality identity \eqref{duality1}, Alladi established the following duality relationship between the smallest and largest prime factors of integers.

\begin{theorem}[{\cite[Theorem 6]{Alladi1977}}] \label{mainthmmuf}
	Let $f:\N\to \C$ be a bounded function. Let $\delta$ be a real number.  Then
	\[
		\lim_{x\to\infty}\frac1x\sum_{n\leq x}f(P^+(n))= \delta
	\]
	if and only if
	\[
		-\lim_{x\to\infty} \sum_{2\leq n\leq x} \frac{\mu(n)f(P^-(n))}{n}=\delta.
	\]
\end{theorem}

Indeed, to obtain Theorem~\ref{mainthmmuf}, Alladi showed that for any function $f$ defined on integers  with $f(1)=0$, 
\begin{equation}\label{eqn_alladi3.8}
\sum_{n\leq x} f(P^+(n))= -x \sum_{n\leq x} \frac{\mu(n)f(P^-(n))}{n}+o(x).
\end{equation}

If $B$ is a finite nonempty set, for any function $a:B\to\C$ on $B$, we define the \textit{Ces\`{a}ro average} of $a$ over $B$ and the \textit{logarithmic average} of $a$ over $B$ respectively by
$$
\BEu{n\in B} a(n) \,\colonequals\, \frac{1}{|B|}\sum_{n\in B} a(n)\qquad\text{and }\qquad
\BEul{n\in B} a(n) \,\colonequals\, \frac{\sum_{n\in B} {a(n)}/{n}}{\sum_{n\in B}{1}/{n}}.
$$
In the following theorem, we establish a logarithmic analogue of  Theorem~\ref{mainthmmuf}.

\begin{theorem}\label{mainthm_log}
Let $f:\N\to \C$ be a bounded function with $f(1)=0$. Then
\begin{equation}
	\lim_{x\to \infty} \Big(\BEul{1\leq n\leq x} f(P^+(n))- \BEu{1\leq n\leq x} f(P^+(n))\Big)=0
\end{equation}
if and only if
\begin{equation}\label{log_alldi}
	\lim_{x\to \infty}\frac{1}{\log x}\sum_{n\leq x} \frac{\mu(n)\log n }{n}f(P^-(n))=0.
\end{equation}
\end{theorem}

\begin{proof}
Let $h(x)=\sum_{n\leq x} \frac1n$. On the one hand, by \eqref{duality1} we have
\begin{align}
	\sum_{n\leq x} \frac{\mu(n)f(P^-(n))}{n} h\big(\frac{x}n\big)&= \sum_{n\leq x} \frac{\mu(n)f(P^-(n))}{n} \sum_{m\leq x/n} \frac1m \nonumber\\
	&= \sum_{mn\leq x} \frac{\mu(n)f(P^-(n))}{mn} \nonumber\\
	&= \sum_{c\leq x} \frac1c \sum_{n\mid c} \mu(n)f(P^-(n)) \nonumber\\
	&= -\sum_{c\leq x} \frac{f(P^+(c))}c. \label{eqn_pf1}
\end{align}

Then by $h(x)=\log x+\gamma+O(1/x)$ where $\gamma$ is Euler's constant, we have
\begin{multline}
	\sum_{n\leq x} \frac{\mu(n)f(P^-(n))}{n} h\big(\frac{x}n\big)= \sum_{n\leq x} \frac{\mu(n)f(P^-(n))}{n} \Big(\log\big(\frac{x}n\big)+\gamma+ O\big(\frac{n}x\big)\Big) \\
	= \log x\sum_{n\leq x} \frac{\mu(n)f(P^-(n))}{n} - \sum_{n\leq x} \frac{\mu(n)\log n }{n}f(P^-(n)) + \gamma\sum_{n\leq x} \frac{\mu(n)f(P^-(n))}{n}+ O(1).\label{eqn_2nd_way}
\end{multline}

By \eqref{eqn_alladi3.8},  the third term in \eqref{eqn_2nd_way} is bounded, since
\begin{equation}\label{Alladi_cor}
	\sum_{n\leq x} \frac{\mu(n)f(P^-(n))}{n}=O(1).
\end{equation}
It follows that
\begin{equation}\label{eqn_pf2}
	\sum_{n\leq x} \frac{\mu(n)f(P^-(n))}{n} h\big(\frac{x}n\big)\\
	= \log x\sum_{n\leq x} \frac{\mu(n)f(P^-(n))}{n} - \sum_{n\leq x} \frac{\mu(n)\log n }{n}f(P^-(n)) +  O(1).
\end{equation}

Combining \eqref{eqn_pf1}
 and \eqref{eqn_pf2} together and using \eqref{eqn_alladi3.8} again, we get that
\begin{align}
	\frac{1}{\log x}\sum_{n\leq x} \frac{\mu(n)\log n }{n}f(P^-(n)) &= \sum_{n\leq x} \frac{\mu(n)f(P^-(n))}{n} + \frac{1}{\log x} \sum_{n\leq x} \frac{f(P^+(n))}n +O\Big(\frac{1}{\log x}\Big) \nonumber\\
	&=\frac{1}{\log x} \sum_{n\leq x} \frac{f(P^+(n))}n  -\frac1x\sum_{n\leq x} f(P^+(n))+o(1), \label{log_alldi_equiv_form}
\end{align}
which completes the proof of Theorem~\ref{mainthm_log}.
\end{proof}

\section{Proof of Theorem~\ref{mainthm_logAlladi}}

To prove \eqref{logAlladi_min} in Theorem~\ref{mainthm_logAlladi}, we cite a result of Kural et al. on the distribution of the largest prime factors of integers.
\begin{theorem}[{\cite[Theorem 3.1]{KuralMcDonaldSah2020}}]
\label{kms}
Let $S$ be a set of primes of natural density $\delta(S)$, then we have
\[
\lim_{x\to\infty}\frac{1}{x}\sum_{\substack{1\leq n\leq x\\ P^+(n)\in S}}1 =\delta(S).
\]
	
\end{theorem}

\begin{proof}[Proof of \eqref{logAlladi_min}]

Suppose the natural density of $S$ over primes is $\delta(S)$. Take $f(n)$ in \eqref{log_alldi} to be the characteristic function of $S$ defined by 
\[
f(n)=
\begin{cases}
	1 & \text{if } n\in S;\\
	0 & \text{if } n\notin S.
\end{cases}
\]

Then by Theorem~\ref{kms}, we have
\[
\lim_{x\to\infty}\BEu{1\leq n\leq x} f(P^+(n))=\lim_{x\to\infty}\frac{1}{x}\sum_{\substack{1\leq n\leq x\\ P^+(n)\in S}}1 =\delta(S).
\]
That is, the natural density of the set $\set{n\in \N: P^+(n)\in S}$ exists. It follows that the logarithmic density of the set $\set{n\in \N: P^+(n)\in S}$ also exists and 
\[
\lim_{x\to\infty}\BEul{1\leq n\leq x} f(P^+(n))= \delta(S).
\]

Thus, by Theorem~\ref{mainthm_log} we have
\[
	\lim_{x\to \infty}\frac{1}{\log x}\sum_{n\leq x} \frac{\mu(n)\log n }{n}f(P^-(n))=0,
\]
which gives \eqref{logAlladi_min}.
\end{proof}

The proof of \eqref{logAlladi_max} is similar to that of \eqref{logAlladi_min}. 

\begin{proof}[Proof of \eqref{logAlladi_max}]
By \cite[(9.15)-(9.16)]{Alladi1982}, we have an estimate similar to \eqref{eqn_alladi3.8} that
\begin{equation}\label{eqn_alladi3.8_max}
\sum_{n\leq x} f(P^-(n))= -x \sum_{n\leq x} \frac{\mu(n)f(P^+(n))}{n}+o(x).
\end{equation}

Then similar to the argument in the proof of Theorem~\ref{mainthm_log}, by \eqref{duality2} and \eqref{eqn_alladi3.8_max} one can obtain that
\begin{equation}\label{max_key_eqn}
	\frac{1}{\log x}\sum_{n\leq x} \frac{\mu(n)\log n }{n}f(P^+(n)) = \frac{1}{\log x} \sum_{n\leq x} \frac{f(P^-(n))}n  -\frac1x\sum_{n\leq x} f(P^-(n))+o(1).
\end{equation}

Take $f(n)$ to be the characteristic function of $S$. By \cite[(9.15)]{Alladi1982}, the following limit exists:
\begin{equation}\label{nd_min}
	\lim_{x\to\infty}\BEu{1\leq n\leq x} f(P^-(n))=\lim_{x\to\infty}\frac{1}{x}\sum_{\substack{1\leq n\leq x\\ P^-(n)\in S}}1,
\end{equation}
which is equal to the natural density of the set $\set{n\in \N: P^-(n)\in S}$. It follows that the logarithmic density of the set $\set{n\in \N: P^-(n)\in S}$ also exists and 
\begin{equation}\label{ld_min}
		\lim_{x\to\infty}\BEul{1\leq n\leq x} f(P^-(n))=\lim_{x\to\infty}\BEu{1\leq n\leq x} f(P^-(n)).
\end{equation}
Thus, \eqref{logAlladi_max} follows immediately by \eqref{max_key_eqn} and \eqref{ld_min}.
\end{proof}

\section{The case on the higher power of logarithm}

In this section, we consider the logarithmic analogue of the PNT with respect to the higher power of logarithm. Let $k\ge1$ be an integer. Write $M(x)\colonequals \sum_{n\leq x} \frac{\mu(n)}{n}$. Then by partial summation, we have
\begin{equation}\label{referee_comment}
	\sum_{n\leq x} \frac{\mu(n)\log^k n}{n} = \log^k(x) M(x) - k\int_1^x \frac{M(t)\log^{k-1} t}{t} dt.
\end{equation}
Since the integral in the right-hand side of \eqref{referee_comment} converges, we have that 
\begin{equation}\label{log_pnt_k}
	\lim_{x\to\infty} \frac1{\log^k x}\sum_{n\leq x} \frac{\mu(n)\log^k n}{n} =0
\end{equation}
holds if and only if 
\begin{equation}\label{pnt_mu}
		\lim_{x\to\infty} M(x)=0.
\end{equation}
This implies that for any integer $k\ge1$, the PNT is equivalent to the assertion that \eqref{log_pnt_k} holds. In this section, we will show that the refinement of \eqref{log_pnt_k} under the same restriction $P^{-}(n)\in S$ or $P^{+}(n)\in S$ as Theorem~\ref{mainthm_logAlladi} still holds. We are deeply grateful to the referee for introducing  the equivalent form \eqref{log_pnt_k} of the PNT and the idea on deriving the equivalence between \eqref{log_pnt_k} and \eqref{pnt_mu} from the identity \eqref{referee_comment} to us, and for suggesting we generalize previous results to the case on the higher power of logarithm. 

In the following, we will follow the spirit of Landau's thesis \cite{Landau1899} to give a proof of the equivalence between \eqref{log_pnt_k} and \eqref{pnt_mu} in an elementary way. For an integer $k\ge 2$, let $d_k(n)$ be the $k$th divisor function definded by 
\[
d_k(n)\colonequals\#\set{(a_1,\dots, a_k)\in \Z^k: a_1\cdots a_k=n, a_1,\dots, a_k \geq1}.
\]
For the case $k=2$, $d_2(n)$ is the classical divisor function. For convenience, we let $d_1(n)=1$ for all $n\ge1$, and $d_0(n)= 1$ if $n=1$ and zero otherwise. Then for any $k\ge1$, we have
\begin{equation}\label{eqn_power_reduced}
	\mu\ast d_k =d_{k-1},
\end{equation}
where the operator $\ast$ denotes the Dirichlet convolution. On the distribution of $d_k(n)$, a well-known elementary theorem states that there exists a polynomial $P_{k}(x)$ of degree $k-1$ and leading coefficient $1/(k-1)!$ such that
\begin{equation}
	\sum_{n\leq x} d_k(n)= x P_{k}(\log x) + O(x^{1-\frac1k}\log^{k-2}x)
\end{equation}
for $k\ge2$, e.g., see \cite[Chapter XII]{Titchmarsh1986}.
Then by partial summation, there exists a polynomial $Q_k(x)$ of degree $k$ and leading coefficient $1/k!$ such that
\begin{equation}\label{divisor_fcn_logaverage}
	\sum_{n\leq x} \frac{d_k(n)}n=  Q_k(\log x) + O(x^{-\frac1k}\log^{k-2}x).
\end{equation}

\subsection{The higher power logarithmic PNT}

By \cite[Lemmas 1 and 2]{Mangoldt1897}, we have
\begin{equation}
	\Big|\sum_{n\leq x} \frac{\mu(n)}{n} \Big| \leq 1
\end{equation}
and 
\begin{equation}\label{thm_Mangoldt}
	\Big|\log x\sum_{n\leq x} \frac{\mu(n)}{n} - \sum_{n\leq x} \frac{\mu(n)\log n}{n} \Big| \leq 3+\gamma
\end{equation}
for all $x\ge1$. This implies that
\begin{align}
	& \sum_{n\leq x} \frac{\mu(n)}{n}=O(1), \\
	& \sum_{n\leq x} \frac{\mu(n)}{n} \log(\frac{x}{n})=O(1). \label{thm_Mangoldt_1}
\end{align}

In the next theorem, we will apply \eqref{eqn_power_reduced} and \eqref{divisor_fcn_logaverage} to generalize  \eqref{thm_Mangoldt_1} to the summation with higher power of logarithm.

\begin{proposition}\label{thm_Mangoldt_k}
	For any integer $k\ge1$, we have
	\begin{equation}\label{thm_Mangoldt_k_estimate}
		\sum_{n\leq x} \frac{\mu(n)}{n} \log^k(\frac{x}{n}) = O(\log^{k-1}x).
	\end{equation}
\end{proposition}

\begin{proof}
We prove \eqref{thm_Mangoldt_k_estimate} by doing induction on $k$. For $k=1$,  \eqref{thm_Mangoldt_k_estimate} follows by  \eqref{thm_Mangoldt_1}. Let $k\geq2$, and suppose \eqref{thm_Mangoldt_k_estimate} holds for  the cases $1,\dots, k-1$. First, by \eqref{eqn_power_reduced} and \eqref{divisor_fcn_logaverage} we observe that
\begin{multline}\label{thm_Mangoldt_k_pfeqn1}
	\sum_{n\leq x} \frac{\mu(n)}{n} \sum_{m\leq \frac{x}{n}} \frac{d_k(m)}{m}  = \sum_{mn\leq x} \frac{ \mu(n) d_k(m) }{mn}  = \sum_{c\leq x} \frac1c \sum_{n\mid c} \mu(n)d_k(\frac{c}{n}) = \sum_{c\leq x} \frac{d_{k-1}(c)}{c} \\= O(\log^{k-1} x). 
\end{multline}

Suppose $Q_k(x)=\sum_{i=0}^kc_ix^i, c_i\in \R$ and $c_k=1/k!$. Then by assumption we have 
$$\sum_{n\leq x} \frac{\mu(n)}{n} \log^i(\frac{x}{n}) = O(\log^{i-1}x), i=1,\dots, k-1.$$
It follows that
\begin{align}
	\sum_{n\leq x} \frac{\mu(n)}{n} \sum_{m\leq \frac{x}{n}} \frac{d_k(m)}{m} & = \frac{1}{k!} \sum_{n\leq x} \frac{\mu(n)}{n} \log^k(\frac{x}{n}) + \sum_{i=1}^{k-1}c_i \sum_{n\leq x} \frac{\mu(n)}{n} \log^i(\frac{x}{n})  \nonumber\\
	&\quad + c_0 \sum_{n\leq x} \frac{\mu(n)}{n}  + O\of{\sum_{n\leq x} \frac{1}{n} (\frac{n}{x})^{\frac1k}\log^{k-2}(\frac{x}{n})} \nonumber\\
	&= \frac{1}{k!} \sum_{n\leq x} \frac{\mu(n)}{n} \log^k(\frac{x}{n}) +  O(\log^{k-2} x).  \label{thm_Mangoldt_k_pfeqn2}
\end{align}
Combing \eqref{thm_Mangoldt_k_pfeqn1} and \eqref{thm_Mangoldt_k_pfeqn2} together, we obtain that
	\begin{equation}
		\sum_{n\leq x} \frac{\mu(n)}{n} \log^k(\frac{x}{n}) = O(\log^{k-1}x).
	\end{equation}
Thus, by induction \eqref{thm_Mangoldt_k_estimate} holds for all $k\ge1$.
\end{proof}

\begin{remark}
	In \cite{Shapiro1950}, Shapiro gave an asymptotic formula for \eqref{thm_Mangoldt_k_estimate}. One may apply the method in the proof of Proposition~\ref{thm_Mangoldt_k} to obtain an asymptotic estimate of \eqref{thm_Mangoldt_k_estimate} as well.
\end{remark}

Now, similar to \eqref{thm_Mangoldt}, we build a relationship between $\sum_{n\leq x} \frac{\mu(n)\log^kn}{n}$ and $\log^k x\sum_{n\leq x} \frac{\mu(n)}{n}$ by doing induction on $k$ for $k\ge1$.

\begin{proposition}\label{thm_Landau_k}
	For any $k\ge1$, we have
	\begin{equation}\label{thm_Landau_k_eqn}
		\sum_{n\leq x} \frac{\mu(n)\log^kn}{n} = \log^k x \sum_{n\leq x} \frac{\mu(n)}{n} + O(\log^{k-1}x).
	\end{equation}
\end{proposition}

\begin{proof}
	For $k=1$, \eqref{thm_Landau_k_eqn} follows by  \eqref{thm_Mangoldt_1}. Let $k\geq2$, and suppose \eqref{thm_Landau_k_eqn} holds for  the cases $1,\dots, k-1$. Then
	\begin{align}
		&\quad \sum_{n\leq x} \frac{\mu(n)}{n} \log^k(\frac{x}{n}) \nonumber\\
		&= \sum_{n\leq x} \frac{\mu(n)}{n}(\log x-\log n)^k = \sum_{i=0}^k \binom{k}{i} (-1)^i \log^{k-i}x \sum_{n\leq x} \frac{\mu(n)\log^i n}{n} \nonumber\\
		&= \log^k x \sum_{n\leq x} \frac{\mu(n)}{n}  + (-1)^k \sum_{n\leq x} \frac{\mu(n)\log^kn}{n}   + \sum_{i=1}^{k-1} \binom{k}{i} (-1)^i \log^{k-i}x \sum_{n\leq x} \frac{\mu(n)\log^i n}{n}\nonumber\\
		&= \log^k x \sum_{n\leq x} \frac{\mu(n)}{n}  + (-1)^k \sum_{n\leq x} \frac{\mu(n)\log^kn}{n} \nonumber\\
		&\quad + \sum_{i=1}^{k-1} \binom{k}{i} (-1)^i \log^{k-i}x \Big(\log^i x \sum_{n\leq x} \frac{\mu(n)}{n} + O(\log^{i-1}x)\Big) \nonumber\\
		&= \Big(\sum_{i=0}^{k-1} \binom{k}{i} (-1)^i \Big) \cdot \log^k x \sum_{n\leq x} \frac{\mu(n)}{n}  + (-1)^k \sum_{n\leq x} \frac{\mu(n)\log^kn}{n} +  O(\log^{k-1}x) \nonumber\\
		&= (-1)^{k+1} \log^k x \sum_{n\leq x} \frac{\mu(n)}{n}  + (-1)^k \sum_{n\leq x} \frac{\mu(n)\log^kn}{n} +  O(\log^{k-1}x).\label{thm_Landau_k_pf1}
	\end{align}
	
	By Proposition~\ref{thm_Mangoldt_k}, we have
	\begin{equation}\label{thm_Landau_k_pf2}
		\sum_{n\leq x} \frac{\mu(n)}{n} \log^k(\frac{x}{n}) = O(\log^{k-1}x).
	\end{equation}	
Then combing \eqref{thm_Landau_k_pf1} and \eqref{thm_Landau_k_pf2} gives us that
	\begin{equation}
		\sum_{n\leq x} \frac{\mu(n)\log^kn}{n} = \log^k x \sum_{n\leq x} \frac{\mu(n)}{n} + O(\log^{k-1}x).
	\end{equation}
Thus, by induction \eqref{thm_Landau_k_eqn} holds for all $k\ge1$.
\end{proof}

By Proposition~\ref{thm_Landau_k}, we find that
\begin{equation}\label{thm_Landau_k_form2}
	\frac{1}{\log^k x}\sum_{n\leq x} \frac{\mu(n)\log^kn}{n} =  \sum_{n\leq x} \frac{\mu(n)}{n} + O\Big(\frac{1}{\log x}\Big).
\end{equation}
From \eqref{thm_Landau_k_form2}, we conclude that \eqref{log_pnt_k} is equivalent to  \eqref{pnt_mu}, which is equivalent to the PNT.

\begin{theorem}
Let $k\ge1$ be any given integer. Then the prime number theorem is equivalent to the assertion that
	\begin{equation}\label{log_pnt_k_rep}
		\lim_{x\to\infty}\frac1{\log^k x}\sum_{n\leq x} \frac{\mu(n)\log^kn}{n} = 0.
	\end{equation}
\end{theorem}

Now, we show that under the same restriction $P^{-}(n)\in S$ or $P^{+}(n)\in S$ as Theorem~\ref{mainthm_logAlladi}, the refinement of \eqref{log_pnt_k_rep} holds for any $k\ge1$. This result can be obtained from the following lemma.

\begin{lemma}\label{lem_logweight}
	If an arithmetic function $a:\N\to\C$ satisfies that
	\begin{equation}
		\sum_{n\leq x} a(n) = o(\log x),
	\end{equation}
	then for any $k\ge2$, we have
	\begin{equation}\label{lem_logweight_eqn}
		\sum_{n\leq x} a(n)\log^{k-1} n = o(\log^k x).
	\end{equation}
\end{lemma}
\begin{proof}
	Let $A(x)=\sum_{n\leq x} a(n)$. Then by partial summation, we have
	\begin{equation}
		\sum_{n\leq x} a(n)\log^{k-1} n = A(x) \log^{k-1} x- (k-1) \int_1^x \frac{A(t)\log^{k-2} t}{t}dt.
	\end{equation}
	
	For any $\ve>0$, since $A(x)=o(\log x)$, there exists some $x_1>1$ such that $|A(x)|\leq \ve \log x$ for all $x\geq x_1$. For such $x_1$, there exists some $x_2>1$ such that
	\begin{equation}\label{mainthm_logAlladi_k_pf1}
		\Big|\int_1^{x_1} \frac{A(t)\log^{k-2} t}{t}dt \Big|\leq \frac{\ve}{k-1} \log^k x
	\end{equation}
	for all $x\ge x_2$. Let $x_0=\max\set{x_1,x_2}$. Then for all $x\ge x_0$, we have \eqref{mainthm_logAlladi_k_pf1} and
	\begin{equation}\label{mainthm_logAlladi_k_pf2}
		\Big|\int_{x_1}^x \frac{A(t)\log^{k-2} t}{t}dt \Big|\leq  \int_{x_1}^x \frac{|A(t)|\log^{k-2} t}{t}dt  \leq  \ve \int_{x_1}^x \frac{\log^{k-1} t}{t}dt < \frac{\ve}{k} \log^k x.
	\end{equation}
	It follows by the triangle inequality that
	\begin{equation}
		\Big|\sum_{n\leq x} a(n)\log^{k-1} n \Big|\leq 3\ve \log^k x
	\end{equation}
	for all $x\ge x_0$. Thus, \eqref{lem_logweight_eqn} holds for any $k\ge2$.
\end{proof}

\begin{theorem}\label{mainthm_logAlladi_k}
	Let $k\ge1$ be an integer.	If $S$ is a set of primes of natural density, then we have
	\begin{equation}\label{logAlladi_min_k}
			\lim_{x\to \infty}\frac{1}{\log^k x}\sum_{\substack{n\leq x\\P^-(n)\in S}} \frac{\mu(n)\log^k n }{n}=0.
	\end{equation}	
Moreover,
	\begin{equation}\label{logAlladi_max_k}
			\lim_{x\to \infty}\frac{1}{\log ^kx}\sum_{\substack{n\leq x\\P^+(n)\in S}} \frac{\mu(n)\log^k n }{n}=0.
	\end{equation}
\end{theorem}

\begin{proof} 
We only show that \eqref{logAlladi_min_k} holds, since the proof of \eqref{logAlladi_max_k} is similar. By Theorem~\ref{mainthm_logAlladi}, \eqref{logAlladi_min_k} holds for $k=1$. For $k\ge2$, we take
$$a(n)=\frac{\mu(n)\log(n) 1_{P^-(n)\in S}}{n},$$
where
\[
1_{P^-(n)\in S}=
\begin{cases}
	1 & \text{if } P^-(n)\in S;\\
	0 & \text{if } P^-(n)\notin S.
\end{cases}
\]
Then by Theorem~\ref{mainthm_logAlladi}, $\sum_{n\leq x}a(n)=o(\log x)$. Notice that
\begin{equation}
	a(n)\log^{k-1}n=\frac{\mu(n)\log^k (n) 1_{P^-(n)\in S} }{n}
\end{equation}
and
\begin{equation}
	\sum_{n\leq x}a(n)\log^{k-1}n = \sum_{\substack{n\leq x\\P^-(n)\in S}} \frac{\mu(n)\log^k n }{n}.
\end{equation}
Thus, using Lemma~\ref{lem_logweight}, we get \eqref{logAlladi_min_k}, as desired.
\end{proof}

\subsection{Higher power logarithmic duality}

In this last section, we generalize the logarithmic duality Theorem~\ref{mainthm_log} to the higher power case. We will use the idea in the proof of Proposition~\ref{thm_Mangoldt_k}. We first generalize Alladi's duality identitis \eqref{duality1} and \eqref{duality2} to a form involving the divisor functions.

\begin{lemma}
Let $f: \N \to\C$ be a function with $f(1)=0$. Then	for any $k\geq1$, we have
	\begin{align}
		[\mu f(P^-(\cdot))]\ast d_k &= - f(P^+(\cdot)) \ast d_{k-1}, \label{duality1_k} \\
		[\mu f(P^+(\cdot))]\ast d_k &= - f(P^-(\cdot)) \ast d_{k-1}.\label{duality2_k}
	\end{align}
	That is,
		\begin{align}
		\sum_{d\mid n} \mu(d)f(P^-(d)) d_k(\frac{n}{d}) & = -\sum_{d\mid n} f(P^+(d)) d_{k-1}(\frac{n}{d}), \\
		\sum_{d\mid n} \mu(d)f(P^+(d)) d_k(\frac{n}{d}) & = -\sum_{d\mid n} f(P^-(d)) d_{k-1}(\frac{n}{d}) 
		\end{align}
		for all $n\ge1$.
\end{lemma}
\begin{proof}
We only show that \eqref{duality1_k} holds, since the proof of \eqref{duality2_k} is similar.
For $k=1$, \eqref{duality1_k} follows by \eqref{duality1}. That is, 
\begin{equation}
	\mu f(P^-(\cdot))]\ast d_1 = - f(P^+(\cdot)) \ast d_{0} =  - f(P^+(\cdot)).
\end{equation}

Let $k\ge2$. Notice that $d_k= d_1\ast d_{k-1}$ for all $k\ge2$, then 
$$[\mu f(P^-(\cdot))]\ast d_k = [\mu f(P^-(\cdot))]\ast (d_1\ast d_{k-1}) =  ([\mu f(P^-(\cdot))]\ast d_1) \ast d_{k-1} = - f(P^+(\cdot)) \ast d_{k-1},$$
which gives \eqref{duality1_k}.
\end{proof}

Let $k\ge1$.  For any function $a:B\to\C$ on  a finite nonempty set $B$, we define the $(k-1)$th \textit{logarithmic average} of $a$ over $B$ by
$$
\BEulk{n\in B} a(n) \,\colonequals\, \frac{\sum_{n\in B} {a(n)\log^{k-1}n}/{n}}{\sum_{n\in B}{\log^{k-1} n}/{n}}.
$$

Now, we show the higher power logarithmic analogue of  Alladi's duality theorem~\ref{mainthmmuf} in the following theorem.  Similar to the proof of Theorem~\ref{mainthm_logAlladi}, one can also derive Theorem~\ref{mainthm_logAlladi_k} from Theorem~\ref{mainthm_log_k}.

\begin{theorem}\label{mainthm_log_k}
Let $k\ge1$. Let $f:\N\to \C$ be a bounded function with $f(1)=0$. Then
\begin{equation}\label{diff_k}
	\lim_{x\to \infty} \Big(\BEulk{1\leq n\leq x} f(P^+(n))- \BEu{1\leq n\leq x} f(P^+(n))\Big)=0
\end{equation}
if and only if
\begin{equation}\label{log_alldi_k}
	\lim_{x\to \infty}\frac{1}{\log^k x}\sum_{n\leq x} \frac{\mu(n)\log^k n }{n}f(P^-(n))=0.
\end{equation}
Moreover,
\begin{equation}
	\lim_{x\to \infty} \Big(\BEulk{1\leq n\leq x} f(P^-(n))- \BEu{1\leq n\leq x} f(P^-(n))\Big)=0
\end{equation}
if and only if
\begin{equation}
	\lim_{x\to \infty}\frac{1}{\log^k x}\sum_{n\leq x} \frac{\mu(n)\log^k n }{n}f(P^+(n))=0.
\end{equation}
\end{theorem}

\begin{proof}

We only prove the first equivalence, since the following argument works as well if one switches $P^-(n)$ and $P^+(n)$. To prove that \eqref{diff_k} is equivalent to \eqref{log_alldi_k}, it suffices to show that
\begin{equation}\label{general_kth_duality}
	\sum_{n\leq x} \frac{\mu(n)f(P^-(n)) \log^k n }{n}= k\sum_{n\leq x} \frac{f(P^+(n)) \log^{k-1} n }{n} - \frac{\log^k x}{x} \sum_{n\leq x} f(P^+(n)) +o(\log^k x)
\end{equation}
holds for all $k\ge1$. We do induction on $k$. For $k=1$, \eqref{general_kth_duality} follows by \eqref{log_alldi_equiv_form}. Let $k\ge2$, and suppose \eqref{general_kth_duality} holds for the cases $1,\dots, k-1$. Then by \eqref{general_kth_duality} and \eqref{Alladi_cor}, we have
\begin{equation}
	\sum_{n\leq x} \frac{\mu(n)f(P^-(n)) \log^i n }{n} =O(\log^i x)
\end{equation}
for all $i=0, 1, \dots, k-1$. It follows by using the binomial theorem that
\begin{equation}\label{log_alldi_k_pf1}
		\sum_{n\leq x} \frac{\mu(n)f(P^-(n))}{n} \log^i (\frac{x}{n})=O(\log^i x)
\end{equation}
for all $i=0, 1, \dots, k-1$. Similar to \eqref{thm_Mangoldt_k_pfeqn1} in the proof of Propostion~\ref{thm_Mangoldt_k}, we observe that
	\begin{align}
		\sum_{n\leq x} \frac{\mu(n)f(P^-(n))}{n} \sum_{m\leq \frac{x}{n}} \frac{d_k(m)}{m} & =\sum_{mn\leq x} \frac{\mu(n)f(P^-(n))d_k(m)}{mn} \nonumber\\
		 & = \sum_{c\leq x} \frac1c \sum_{n\mid c}\mu(n)f(P^-(n))d_k(\frac{c}{n}) \nonumber\\
		 & = - \sum_{c\leq x} \frac1c  \sum_{n\mid c} f(P^+(n))d_{k-1}(\frac{c}{n})  \nonumber\\
		 & = - \sum_{n\leq x} \frac{f(P^+(n))}{n} \sum_{m\leq \frac{x}{n}} \frac{d_{k-1}(m)}{m} \nonumber\\
		 &=- \frac1{(k-1)!}\sum_{n\leq x} \frac{f(P^+(n))}{n} \log^{k-1}(\frac{x}{n}) + O(\log^{k-1} x). \label{log_alldi_k_total}
	\end{align}

Then similar to \eqref{thm_Mangoldt_k_pfeqn2} in the proof of Propostion~\ref{thm_Mangoldt_k}, using \eqref{divisor_fcn_logaverage} and \eqref{log_alldi_k_pf1} we get that
\begin{equation}\label{log_alldi_k_total_1st_term}
	\sum_{n\leq x} \frac{\mu(n)f(P^-(n))}{n} \sum_{m\leq \frac{x}{n}} \frac{d_k(m)}{m} =\frac1{k!}\sum_{n\leq x} \frac{\mu(n)f(P^-(n))}{n} \log^k (\frac{x}{n}) + O(\log^{k-1} x).
\end{equation}

Now, expanding out the first term of \eqref{log_alldi_k_total_1st_term}, using the assumption \eqref{general_kth_duality} for the cases $1,\dots, k-1$,   we get that
\begin{align}
	&\quad \sum_{n\leq x} \frac{\mu(n)f(P^-(n))}{n} \log^k (\frac{x}{n}) \nonumber\\
	& = 	\sum_{n\leq x} \frac{\mu(n)f(P^-(n))}{n} (\log x -\log n)^k  = \sum_{i=0}^k \binom{k}{i} (-1)^i \log^{k-i}x \sum_{n\leq x} \frac{\mu(n)f(P^-(n)) \log^i n }{n}  \nonumber\\
	 & = \log^k x \sum_{n\leq x} \frac{\mu(n)f(P^-(n))}{n} + (-1)^k \sum_{n\leq x} \frac{\mu(n)f(P^-(n)) \log^k n }{n} \nonumber\\
	 & \quad + \sum_{i=1}^{k-1} \binom{k}{i} (-1)^i \log^{k-i}x \Big(i\sum_{n\leq x} \frac{f(P^+(n)) \log^{i-1} n }{n} - \frac{\log^i x}{x} \sum_{n\leq x} f(P^+(n)) +o(\log^i x)\Big)\nonumber\\
	 & = - \frac{\log^k x}{x} \sum_{n\leq x} f(P^+(n)) + (-1)^k \sum_{n\leq x} \frac{\mu(n)f(P^-(n)) \log^k n }{n}  \nonumber\\
	 &\quad + \sum_{n\leq x} \frac{f(P^+(n))}{n} \sum_{i=1}^{k-1} \binom{k}{i} (-1)^i {i} \log^{k-i}x  \log^{i-1} n  \nonumber\\
	 &\quad - \sum_{i=1}^{k-1} \binom{k}{i} (-1)^i  \frac{\log^k x}{x} \sum_{n\leq x} f(P^+(n) + o(\log^k x) \nonumber\\
	 &= (-1)^k \frac{\log^k x}{x} \sum_{n\leq x} f(P^+(n))  + (-1)^k \sum_{n\leq x} \frac{\mu(n)f(P^-(n)) \log^k n }{n} \nonumber\\
	 & \quad -k\sum_{n\leq x} \frac{f(P^+(n))}{n} \log^{k-1}(\frac{x}{n}) - (-1)^kk \sum_{n\leq x} \frac{f(P^+(n))\log^{k-1} n}{n}  + o(\log^k x). \label{log_alldi_k_LHS_0}
\end{align}
In the last line of \eqref{log_alldi_k_LHS_0}, we have used the following two identities
\begin{align}
\sum_{i=1}^{k-1} \binom{k}{i} (-1)^i & =-1-(-1)^k, \\
\sum_{i=1}^{k-1} \binom{k}{i} (-1)^i {i} a^{k-i}t^{i-1}&= -k(a-t)^{k-1} -(-1)^k k t^{k-1}, \label{binomial_thm_derivative}
\end{align}
where $a=\log x, t=\log n$. And \eqref{binomial_thm_derivative} can be obtained by taking derivative on $t$ in the binomial expansion of $(a-t)^k$. 

Dividing both sides of \eqref{log_alldi_k_LHS_0} by $k!$, we get that
\begin{multline}
	\frac1{k!}\sum_{n\leq x} \frac{\mu(n)f(P^-(n))}{n} \log^k (\frac{x}{n}) = \frac{(-1)^k}{k!} \Big( \sum_{n\leq x} \frac{\mu(n)f(P^-(n)) \log^k n }{n} + \frac{\log^k x}{x} \sum_{n\leq x} f(P^+(n)) \\
	-k \sum_{n\leq x} \frac{f(P^+(n))\log^{k-1} n}{n}\Big)-\frac1{(k-1)!}\sum_{n\leq x} \frac{f(P^+(n))}{n} \log^{k-1}(\frac{x}{n}) + o(\log^k x). \label{log_alldi_k_LHS}
\end{multline}

Combining \eqref{log_alldi_k_total} and \eqref{log_alldi_k_LHS} together, we obtain  \eqref{general_kth_duality} for the case $k$. Thus,  \eqref{general_kth_duality} holds for all $k\ge1$ by induction.
\end{proof}

\section*{Acknowledgments}

The author is deeply grateful to the referee for a very careful review and very insightful suggestions, which improve this paper a lot; in particular, for the suggestion of incorporating Theorems \ref{mainthm_logAlladi_k} and  \ref{mainthm_log_k}  to the original version of the paper.  The author would also like to thank Shaoyun Yi for the discussions on the earlier draft of this paper.

\end{document}